\documentclass[reqno,english]{amsart}
\usepackage{amssymb,latexsym,amsfonts,verbatim,amscd,mathrsfs,array}
\usepackage{amsmath,color}
\usepackage{amsthm}
\usepackage{epsfig,pdfsync}
\usepackage{graphicx}
\usepackage{hyperref}
\usepackage{titletoc}
\numberwithin{equation}{section}
\newtheorem{theorem}{Theorem}[section]
\newtheorem{proposition}[theorem]{Proposition}
\newtheorem{lemma}[theorem]{Lemma}

\newtheorem{Theorem}{Theorem}[section]

\theoremstyle{definition}

\newcommand{\D}{\Delta}

\newcommand{\R}{\mathbb{R}}
\newcommand{\ve}{\varepsilon}

\newcommand{\B}{\mathbb{R}^2}
\newcommand{\vf}{\mathbf{v}}
\newcommand{\x}{\mathbb{X}}
\newcommand{\w}{\widetilde}

\thanks{The first author is supported by the SNSF   Grant No. P2BSP2-172064.}
\thanks{ The  second author  is partially supported by NSERC.}
\thanks{The third author is partially supported by NSFC No.11801550 and NSFC No.11871470.}

\begin{document}
\title{On the general Toda system with multiple singular points}
\author[A. Hyder]{Ali Hyder}
\address{\noindent Department of Mathematics, University of British Columbia,
Vancouver, B.C., Canada, V6T 1Z2}
\email{ali.hyder@math.ubc.ca}

\author[J. Wei]{Juncheng Wei}
\address{\noindent Department of Mathematics, University of British Columbia,
Vancouver, B.C., Canada, V6T 1Z2}
\email{jcwei@math.ubc.ca}

\author[W. Yang]{ Wen Yang}
\address{\noindent Wen ~Yang,~Wuhan Institute of Physics and Mathematics, Chinese Academy of Sciences, P.O. Box 71010, Wuhan 430071, P. R. China}
\email{wyang@wipm.ac.cn}

\begin{abstract}
In this paper, we consider the following elliptic Toda system associated to a general simple Lie algebra with multiple singular sources
\begin{equation*}
\begin{cases}
-\Delta w_i=\sum_{j=1}^na_{i,j}e^{2w_j}+2\pi\sum_{\ell=1}^m\beta_{i,\ell}\delta_{p_\ell}
\quad&\mbox{in}\quad\mathbb{R}^2,\\
\\
w_i(x)=-2\log|x|+O(1)~\mbox{as}~|x|\to\infty,\quad &i=1,\cdots,n,
\end{cases}
\end{equation*}
where $\beta_{i,\ell}\in[0,1)$. Under some suitable assumption on $\beta_{i,\ell}$ we establish the existence and non-existence results. This paper generalizes Luo-Tian's \cite{Luo-Tian} and Hyder-Lin-Wei's \cite{hlw} results to the general Toda system.
\end{abstract}

\maketitle

\section{Introduction}
In this paper, we shall consider the following singular Toda system with multiple singular sources
\begin{equation}
\label{1.su}
-\Delta w_i
=\sum_{j=1}^na_{i,j}e^{2w_j}+2\pi\sum_{\ell=1}^m\beta_{i,\ell}\delta_{p_\ell}\quad\mbox{in}\quad\mathbb{R}^2,
\end{equation}
where $a_{i,j}$ is the Cartan matrix associated to a simple Lie algebra, $\beta_{i,l}\in[0,1)$, $p_1,\cdots,p_m$ are distinct points in $\B$ and $\delta_{\ell}$ denotes the Dirac measure at $p_\ell,~\ell=1,\cdots,m.$

When the Lie algebra is $\mathbf{A}_1=\mathfrak{sl}_2$, \eqref{1.su} becomes the Liouville equation
\begin{equation}
\label{1.liouville}
\Delta u+e^{2u}=-2\pi\sum_{\ell=1}^m\gamma_{\ell}\delta_{p_\ell}\quad\mbox{in}\quad\B.
\end{equation}
The Toda system \eqref{1.su} and the Liouville equation \eqref{1.liouville} arise in many physical and geometric problem. On geometric side, Liouville equation is related to the Nirenberg problem of finding a conformal metric with  prescribed Gaussian curvature if $\{p_1,\cdots,p_m\}=\emptyset$, and the existence of the same curvature metric of problem \eqref{1.liouville} with conical singularities at $\{p_1,\cdots,p_m\}$. When the Lie algebra is $\mathbf{A}_n$, the Toda system \eqref{1.su} is closely related to holomorphic curves in projective spaces \cite{Doliwa} and the Pl\"ucker formulas \cite{gh}, while the periodic Toda systems are related to harmonic maps \cite{M-Guest}. In physics, the Toda system is a well-known integrable system and closely related to the $\mathcal{W}$-algebra in conformal field theory, see \cite{bfr,frrtw} and references therein. Liouville equation and Toda system also played an important role in Chern-Simons gauges theory. For example, $\mathbf{A}_n$ (n=2) Toda system governs the limit equations as physical parameters tend to $0$ and is used to explain the physics of high temperature, we refer the readers to \cite{Dunne,Yang,Yang1} for more background on it.

For the Liouville equation \eqref{1.liouville}, Chen and Li \cite{CL} classified all  solutions when there is no singular source provided the total integration of $e^{2u}$ on $\mathbb{R}^2$ is finite. Under the same integrability condition, Prajapat and Tarantello \cite{PT} completed the classification with one singular point. The question of conformal metrics with multiple conical singularities has been widely studied using various viewpoints. When $m=2$, equation \eqref{1.liouville} is equivalent to a Mean Field equation on $\mathbb{S}^2$ with three singularities, which can be chosen as $0,1$ and $\infty$ by M\"obius transformation, Eremenko's \cite{Eremenko} work gives the necessary and sufficient condition for the existence of a conformal metric of constant Gaussian curvature by studying the monodromy of the corresponding second order hypergeometric equation. For equation \eqref{1.liouville} with general $m\geq3$, Troyanov \cite{Tr} proved that there exists a solution provided $\gamma_{\ell},~\ell=1,2,\cdots,m$ satisfies the following condition
\begin{equation}
\label{1.troyanov}
0<2-\sum_{\ell=1}^m\gamma_{\ell}<2\min\{1,\min_{1\leq \ell\leq m}(1-\gamma_{\ell})\}.
\end{equation}
Later on, Luo-Tian \cite{Luo-Tian} proved that if the condition $0<\gamma_{\ell}<1$  is satisfied for $1\leq \ell\leq m$, then \eqref{1.troyanov} is necessary and sufficient, and the solution is unique. Precisely, we state it as the following theorem
\begin{Theorem}[\cite{Luo-Tian}]
\label{th.lt}
Let $m\geq 3$ and $p_1,\cdots,p_m$ be $m$ distinct points in $\B$. Then there exists a solution to \eqref{1.liouville} verifying the following behavior
\begin{align}
\label{1.ltbehavior}\left\{ \begin{array}{ll}
u(x)=-\gamma_\ell\log|x-p_\ell|+\mbox{bounded continuous function}\quad &\mbox{around each}~p_{\ell}, \\  \rule{0cm}{.5cm}
u(x)=-2\log|x|+\mbox{bounded continuous function} &\mbox{as}~|x|\to\infty,\\  \rule{0cm}{.5cm}
\gamma_\ell\in(0,1),&\ell=1,\cdots,m, \end{array}\right.
\end{align}
if and only if \eqref{1.troyanov} holds. Moreover, the solution is unique.
\end{Theorem}

We rewrite \eqref{1.troyanov} as follows
\begin{equation}
\label{1.condition-1}
\sum_{\ell=1}^m\gamma_\ell<2\quad\mathrm{and}\quad\sum_{\ell\neq j}\gamma_\ell>\gamma_j~\mbox{for every}~
j=1,\cdots,m.
\end{equation}
It is interesting to get a counterpart result of Theorem \ref{th.lt} for the Toda system of general simple Lie algebra. When the Lie algebra is $\mathfrak{sl}_3$, the first two authors of this paper and Lin \cite{hlw} deduce an existence result of \eqref{1.su} provided $\beta_{i,\ell},~i=1,2,~\ell=1,\cdots,m$ satisfies
\begin{equation}
\label{1.condition-2}
3(1+\beta_{i,j})<2\sum_{\ell=1}^m\beta_{i,\ell}+\sum_{\ell=1}^m\beta_{3-i,\ell},~
\sum_{l=1}^3\beta_{i,\ell}<2~\mbox{for}~i=1,2,~j=1,\cdots,m.
\end{equation}
While the authors also show that an equivalent condition of \eqref{1.condition-1} for $\mathbf{A}_2$ Toda system is not sufficient for the existence of solutions to \eqref{1.liouville} satisfying the behavior \eqref{1.ltbehavior}.

In this paper, we shall study the same problem for general Toda system. Precisely, we shall consider
the existence and non-existence of  solutions $(w_1,\cdots,w_n)$ to \eqref{1.su} satisfying
\begin{align}
\label{1.asy}
%\begin{cases}
\left\{ \begin{array}{ll}
w_i(x)=-\beta_{i,\ell}\log|x-p_\ell|+h_{i,\ell}\quad &\mbox{around each point}~p_\ell,\\  \rule{0cm}{0.4cm}
w_i(x)=-2\log|x|+h_{i,m+1}\quad &\mbox{as}~|x|\to+\infty,\\ \rule{.0cm}{0.4cm}
h_{i,\ell}(x)~\mbox{is  continuous in a neighbourhood of}~p_\ell,\end{array}\right.
%\end{cases}
\end{align}
for $i=1,\cdots,n$ and $\ell=1,\cdots,m$ and $h_{i,m+1}$ is bounded outside a compact set. We set
\begin{equation*}
%\label{1.decompose}
u_i(x)=w_i(x)+\sum_{\ell=1}^m\beta_{i,\ell}\log|x-p_\ell|,\quad i=1,\cdots,n.
\end{equation*}
We see that $w_i$ solves \eqref{1.su} if and only if $u_i$ solves
\begin{align}
\label{1.u}
%\begin{cases}
\left\{ \begin{array}{ll}-\Delta u_i=\sum_{j=1}^na_{i,j}K_je^{2u_j}\quad\mbox{in}\quad \B,\\ \rule{0cm}{.5cm}
K_i(x):=\prod_{\ell=1}^m\frac{1}{|x-p_{\ell}|^{2\beta_{i,\ell}}}.
%\end{cases}
\end{array}\right.
\end{align}
The condition \eqref{1.asy} in terms of $u_i$ is
\begin{equation}
\label{1.newasy}
u_i(x)=-{\beta_i} \log|x|+\mbox{ a bounded continuous function on}~B_{1}^c,
\end{equation}
where %$R_0$ is chosen such that $\{p_1,\cdots,p_\ell\}\subset B_{R_0}$ and
\begin{equation}
\label{1.betai}
\beta_i:=2-\sum_{\ell=1}^m\beta_{i,\ell}.
\end{equation}

Our first result of this paper is on the existence of solutions to  \eqref{1.u}:
\begin{theorem}
\label{th1.1}
Let $m\geq 3$. Suppose $\{\beta_{i,\ell},~i=1,\cdots,n,~\ell=1,\cdots,m\}$ satisfies
\begin{equation}
\label{1.condition}
2\sum_{j=1}^na^{i,j}-1+\beta_{i,\ell}<\sum_{j=1}^n\sum_{s=1}^ma^{i,j}\beta_{j,s},\quad\forall i=1,\cdots,n,~\ell=1,\cdots,m,
\end{equation}
where $(a^{i,j})_{n\times n}$ is the inverse matrix of $(a_{i,j})_{n\times n}.$ Then given any $m$ distinct points $\{p_{\ell}\}_{\ell=1}^m\subset\B$ there exists a continuous solution $(u_1,\cdots,u_n)$ to \eqref{1.u} satisfying \eqref{1.newasy} with $\beta_i$ as in \eqref{1.betai}.
\end{theorem}

We notice that when $n=2$ and $(a_{i,j})$ is Cartan matrix for $\mathbf{A}_2$, then \eqref{1.condition} is equivalent to the condition \eqref{1.condition-2}. Next, we shall show that the equivalent condition proposed by Luo-Tian for single Liouville equation could not work for general Toda systems, namely a condition of the following form can not guarantee the existence of solutions to \eqref{1.u} - \eqref{1.betai}
\begin{equation}
\label{1.ltcondition}
\sum_{s=1,~s\neq \ell}^m\beta_{i,s}>\max_i\beta_{i,\ell}\quad\mbox{for every}\quad i=1,\cdots,n,~\ell=1,\cdots,m.
\end{equation}

The second result of this paper is the following
\begin{theorem}
\label{th1.2}
There exist a tuple of points $\{p_{\ell}\}_{\ell=1}^m\subset\B$ and $\beta_{i,\ell},~i=1,\cdots,n,~\ell=1,\cdots,m$ satisfying \eqref{1.ltcondition}, such that equation \eqref{1.u} has no solution satisfying \eqref{1.newasy}.
\end{theorem}

Let us close the introduction by mentioning the idea used in the proof of Theorem \ref{th1.2}. Our main tools are the induction method and Pohozaev identity. Based on the previous result, we already find points such that \eqref{1.u} has no solution when the coefficient matrix is $\mathbf{A}_2$ and $\beta_{i,\ell}$ satisfies \eqref{1.ltcondition}. This is the starting point of our approach. By assuming the non-existence result of a low rank Toda system, we  obtain that the non-existence result also holds for a higher rank (with rank plus one) by choosing  suitable points. The crucial thing of our argument is to exclude a blow-up phenomena for the higher rank Toda system, where the Pohozaev identity plays an important role. The blow-up phenomena for a general Toda system is very complicated, one of the fundamental issue concerning \eqref{1.u} is the computation of the local mass at the blow-up point. Until now, we can only compute it for $\mathbf{A}_n,~\mathbf{B}_n$,~$\mathbf{C}_n$ and $\mathbf{G}_2$, see \cite{lyz}.

This paper is organized as follows. We study the existence result (Theorem \ref{th1.1}) and non-existence result (Theorem \ref{th1.2}) in sections 2 and 3 respectively. In the last section, we present all the necessary lemmas and facts, including the Cartan matrix of all the simple Lie algebra and their inverse matrices.

\bigskip
\section{Proof of Theorem \ref{th1.1}}
In this section we shall prove Theorem \ref{th1.1}. We notice that if $u_1,\cdots,u_n$ is a continuous solution to \eqref{1.u}-\eqref{1.newasy} with $\beta_{i,\ell}<1$ for all $i=1,\cdots,n$ and $\ell=1,\cdots,m$, then
$$K_i(x)e^{2u_i}\in L^1(\B)$$
and $u_i$ has the following representation formula
\begin{equation*}
%\label{2.rep}
u_i(x)=\frac{1}{2\pi}\sum_{j=1}^na_{i,j}\int_{\B}\log\left(\frac{1}{|x-y|}\right)K_j(y)e^{2u_j(y)}dy+c_i,
~i=1,\cdots,n,
\end{equation*}
for some $c_i\in\R$. Moreover, using the asymptotic behavior \eqref{1.newasy}, we have
\begin{equation*}
%\label{2.relation}
\sum_{j=1}^na_{i,j}\int_{\B}K_je^{2u_j}dx=2\pi\beta_i,\quad i=1,\cdots,n,
\end{equation*}
that is
\begin{equation}
\label{2.rel-1}
\int_{\B}K_ie^{2u_i}=2\pi\bar\beta_i,\quad \bar\beta_i=\sum_{j=1}^na^{i,j}\beta_j,\quad i=1,\cdots,n.
\end{equation}
Then Theorem \ref{th1.1} is equivalent to the existence of solutions $(u_1,\cdots,u_n)$ to \eqref{1.u} satisfying \eqref{2.rel-1} and $\bar\beta_i$ verifies
\begin{equation}
\label{2.rel-2}
\bar\beta_i>0,\quad \bar\beta_i<1-\beta_{i,\ell}\quad \mbox{for every}~i=1,\cdots,n,~\ell=1,\cdots,m.
\end{equation}

As the paper \cite{hlw}, we shall use a fixed point argument to prove the existence.  To set up our argument, we introduce the following functional space
\begin{equation*}
%\label{2.fun}
\mathbb{X}=\underbrace{C_0(\B)\times\cdots\times C_0(\B)}_n,\quad
\|\mathbf{v}\|_\mathbb{X}=:\max_i\left\{\|v_i\|_{L^\infty(\B)}~\mbox{for}~\mathbf{v}\in\mathbb{X}\right\},
\end{equation*}
where   ${\bf{v}}=(v_1,\dots,v_n) $, and  $C_0(\B)$ denotes the space of continuous functions vanishing at infinity. We fix $u_0\in C^\infty(\B)$ such that
\begin{equation*}
%\label{2.u}
u_0(x)=-\log|x|\quad\mbox{on}\quad B_{1}^c. %\footnote{Rigorously, we set $u_0(x)=-\log|x|$ on $B_{R_0}^c$, where $R_0$ is chosen such that $\{p_\ell\}_{\ell=1}^m\subset B_{R_0}$. For simplicity we take $R_0=1$.}
\end{equation*}
For $v\in C_0(\B)$ we set $c_{i,v}$ to be the unique number such that
\begin{equation*}
%\label{2.civ}
\int_{\R^2}\bar{K_i}e^{2(v+c_{i,v})}dx=2\pi\bar\beta_i,\quad \bar K_i:=K_ie^{2\beta_i u_0},\quad i=1,\cdots,n,
\end{equation*}
where $\bar\beta_i$ is defined in \eqref{2.rel-1}. Now we define
$$\mathcal{T}:\mathbb{X}\to\mathbb{X},~(v_1,\cdots,v_n)\mapsto(\bar v_1,\cdots,\bar v_n),$$
where we have set
\begin{equation}
\label{2.barv}
\bar v_i(x):=\frac{1}{2\pi}\sum_{j=1}^na_{i,j}\int_{\R^2}\log\left(\frac{1}{|x-y|}\right)\bar K_j(y)e^{2(v_j(y)+c_{j,v_j})}dy-\beta_i u_0(x),\quad i=1,\cdots,n.
\end{equation}
As $\beta_i=\sum_{j=1}^na_{i,j}\bar\beta_j$, for $x\in B_{1}^c$, \eqref{2.barv} can be written as
\begin{equation*}
%\label{2.barv-1}
\bar v_i(x)=\frac{1}{2\pi}\sum_{j=1}^na_{i,j}\int_{\R^2}\log\left(\frac{|x|}{|x-y|}\right)\bar K_j(y)e^{2(v_j(y)+c_{j,v_j})}dy,\quad i=1,\cdots,n.
\end{equation*}
Using the fact that
$$\bar K_i=K_ie^{2\beta_iu_0}=O(|x|^{-4})\quad\mbox{for}~|x|~\mbox{large},$$
one can prove that
$$\bar v_i(x)\to 0~\mbox{as}~ |x|\to\infty,\quad i=1,\cdots,n.$$
Moreover, the operator $\mathcal{T}$ is compact, see \cite[Lemma 4.1]{HMM}.

To find a fixed point of the map $\mathcal{T}$, it suffices to show that $\mbox{deg}(I-\mathcal{T},\mathbb{X},0)\neq0$. We shall use a homotopy type argument to prove the latter fact. In our homotopy type argument, we need the result below.
\begin{proposition}
\label{pr2.1}
There exists a constant $C>0$ such that
\begin{equation*}
\|\mathbf{v}\|_{\mathbb{X}}\leq C~ \mbox{for every}~ (\mathbf{v},t)\in\mathbb{X}\times[0,1]~  \mbox{satisfying}~
\vf=t\mathcal{T}(\vf).
\end{equation*}
\end{proposition}

\begin{proof}
We assume by contradiction that the result is false, then there exists
$$\vf^k=(v_1^k,\cdots,v_n^k)\quad\mbox{and}\quad t^k\in(0,1]$$
with
$$\vf^k=t^k\mathcal{T}(\vf^k)\quad\mbox{and}\quad \|\vf^k\|_\mathbb{X}\to\infty.$$
We set
\[\psi_i^k:=v_i^k+c_i^k,\quad c_i^k=c_{i,v_i^k}+\frac12\log t^k.\]
Then we have
\begin{equation*}
%\label{2.psi}
\psi_i^k(x)=\frac{1}{2\pi}\int_{\B}\sum_{j=1}^n\log\left(\frac{1}{|x-y|}\right)a_{i,j}\bar K_j(y)e^{2\psi_{j}^k(y)}dy-t^k\beta_iu_0(x)+c_i^k,\quad i=1,\cdots,n.
\end{equation*}
For $|x|\geq 1$ the above equation can be written as
\begin{equation*}
%\label{2.psi-1}
\psi_i^k(x)=\frac{1}{2\pi}\int_{\B}\sum_{j=1}^n\log\left(\frac{|x|}{|x-y|}\right)a_{i,j}\bar K_j(y)e^{2\psi_{j}^k(y)}dy+c_i^k,\quad i=1,\cdots,n.
\end{equation*}
Next we claim that
\begin{equation}
\label{2.claim}
\max_i\left\{\sup\psi_i^k(x)\right\}\xrightarrow{k\to\infty}\infty.
\end{equation}
Indeed if \eqref{2.claim} is not true, we can use the Green representation \eqref{2.barv} together with $\max_i\sup\psi_i^k(x)\leq C$   to obtain that $\|\vf\|_{\x}\leq C$, a contradiction.  Thus \eqref{2.claim} holds. Without loss of generality we may assume that  \footnote{Even though the equation satisfied by $\psi^k_i$ for $2\leq i\leq n-1$ looks slightly different form the one $\psi_1^k$, the proof is same.} \[\sup\psi_1^k(x)=\max_i\left\{\sup\psi_i^k(x)\right\}.\]
Let $x^k\in\B$ be a point such that
\begin{equation*}
\sup\psi_1^k(x)<\psi_1^k(x^k)+1.
\end{equation*}
If $x^k$ is bounded then, up to a subsequence, $x^k\to x^\infty.$ We consider the following two cases.

\indent Case 1. $|x^k|$ is uniformly bounded.

We notice that $\psi_i^k,~i=1,\cdots,n$ satisfies the following equation
\begin{equation*}
%\label{2.eqpsi}
\Delta\psi_i^k+\sum_{j=1}^na_{i,j}\bar K_je^{2\psi_j^k}{+t^k\beta_i\D u_0}=0.
\end{equation*}
Let
\[\mathcal{S}=\{p\in \R^2\mid \max_i\psi_i^k(p_k)\to\infty~\mbox{ for some }~p_k\to p\}.\]
For  $p\in\mathcal{S}$, we define
\begin{equation*}
%\label{2.local}
\sigma_i(p)=\frac{1}{2\pi}\lim_{r\to0}\lim_{k\to\infty}\int_{B_r(p)}\bar K_ie^{2\psi_i^k},\quad i=1,\cdots,n.
\end{equation*}
By Lemma \ref{le4.1}, we have
\begin{equation}
\label{2-i}
\sigma_i(p)\geq \mu_i(p)\quad\mbox{holds at least for one}~i\in\{1,\cdots,n\},
\end{equation} where
\begin{align*}
\mu_i(p):=\left\{ \begin{array}{ll}
1&\quad\text{if }p\not\in\{p_1,\dots,p_m\},\\ \rule{0cm}{.4cm}
1-\beta_{i,\ell}&\quad\text{if }p=p_\ell,\quad \ell=1,\dots,m.
\end{array}\right.
\end{align*}
Using \eqref{2-i} and the fact
\[\frac{1}{2\pi}\int_{\B}\bar K_ie^{2\psi_i^k}=\bar\beta_i,\]
we have for some $i\in\{1,\dots,m\}$
\begin{equation*}
%\label{2.comparison}
\bar\beta_i\geq \sigma_i\geq \mu_i\geq \min_{\ell}\{1,1-\beta_{i,\ell}\}.
\end{equation*}
This contradicts to \eqref{2.rel-2}. Thus $\{|x_k|\}$ is unbounded.\\

\noindent Case 2. $|x^k|\to\infty$.

We set
\begin{equation*}
\tilde\psi_i^k(x)=\psi_i^k(\frac{x}{|x|^2}),\quad \tilde{K}_i(x)=\frac{1}{|x|^4}\bar K_i(\frac{x}{|x|^2})
\quad\mathrm{in}\quad\B\setminus\{0\},~i=1,\cdots,n,
\end{equation*}
and we extend them continuously at the origin. Then $\tilde\psi_i^k$ satisfies
\begin{equation*}
%\label{2.kelvin}
\Delta\tilde\psi_i^k+\sum_{j=1}^na_{i,j}\tilde K_je^{2\tilde\psi_i^k}=0\quad\mathrm{in}\quad B_{1}.
\end{equation*}
Since $\tilde K_i$ is continuous,  $\tilde{K}_i(0)>0,~i=1,\cdots,n$ and
\begin{equation*}
\tilde\psi_i^k(\tilde x^k)\to\infty,\quad \tilde x^k=\frac{x^k}{|x^k|^2}\to0,
\end{equation*}
repeating the arguments of Case 1, we get
\begin{equation*}
\bar\beta_i\geq\frac{1}{2\pi}\int_{B_{1}}\tilde K_je^{2\tilde\psi_i^k}\geq 1\geq1-\beta_{i,\ell}\quad\mbox{for any}~\ell=1,\cdots,m.
\end{equation*}
Contradiction arises again. Therefore there is no blow up for $\{\psi_i^k\}$, and we finish the proof.
\end{proof}

\begin{proof}[Proof of Theorem \ref{th1.1}.]
By Proposition \ref{pr2.1}, we get that
\begin{equation*}
\mbox{deg}(I-\mathcal{T},\mathbb{X},0)=\mbox{deg}(I,\mathbb{X},0)=1.
\end{equation*}
From which we derive the existence of solution to \eqref{1.u} - \eqref{1.betai}.
\end{proof}

\bigskip
\section{Proof of Theorem \ref{th1.2}}
In this section  we shall prove Theorem \ref{th1.2}. We provide the details for $\mathbf{A}_n$ only, and state the differences for other cases at the end of this section. Before the proof, we make the following preparation. For $n\geq 2$ we consider the following tuple of positive numbers:
\begin{equation*}
\mathcal{B}=\left\{b_1,\cdots,b_{4n+1}\in(0,1)\mid\{b_i\}_{i=1}^{4n+1}~\mbox{satisfies the assumption}~\mathcal{D}\right\},
\end{equation*}
where the assumption $\mathcal{D}$ is:
\begin{equation*}
%\label{3.ass}
\mbox{Assumption}~\mathcal{D}: \begin{cases}
\mbox{(d1)}~ &\sum_{i=1}^5b_i=2,~b_2=b_3<\frac12 b_1,~b_4=b_5,~2b_1+b_4<2,\\
\mbox{(d2)}~ &\sum_{i=1}^4b_{4l-3+i}=2,~l=2,\cdots,n,\\
\mbox{(d3)}~ &b_{4l-2}=b_{4l-1}=b_{4l},~l=2,\cdots,n,\\
\mbox{(d4)}~ &n^2b_{4i+1}<\frac{1}{400},~i=1,\cdots,n,\\
\mbox{(d5)}~ &b_4=b_9.
\end{cases}
\end{equation*}
Let us point out that the set $\mathcal{B}$ is not empty, we can choose
\begin{equation*}
%\label{3.example}
\begin{aligned}
&b_1=1-\frac23\ve,\quad b_2=b_3=\frac12-\frac23\ve,\quad b_4=b_5=\ve,\\
&b_{4l-2}=b_{4l-1}=b_{4l}=\frac23-\frac{1}{3}\ve,
%\frac23-\frac{2}{3n(n-1)}\ve,
~b_{4l+1}=\ve,~l=2,\cdots,n.
\end{aligned}
\end{equation*}
Then it is easy to see that the above $\{b_i\}_{i=1}^{4n+1}$ satisfies the assumption $\mathcal{D}$ provided $\ve<\frac{1}{400n^2}$. From (d1) and (d4), we  see that
\begin{equation}
\label{3.con1}
b_4+b_1>1\quad\mbox{and}\quad b_4+b_2=b_4+b_3<1,
\end{equation}
and
\begin{equation}
\label{3.con2}
b_{4i+1}<\frac{1}{50},\quad i=1,\cdots,n.
\end{equation}

We shall show a non-existence result to the Toda system \eqref{1.u} with $m=3n+1$ and $\beta_{i,\ell}$ satisfying the following:
\begin{equation}
\label{3.con3}
\begin{aligned}
&\beta_{i,3i-2+\ell}=b_{4i-3+\ell},~i=2,\cdots,n,~\ell=1,2,3,\\
&\beta_{1,\ell}=b_\ell,~\ell=1,\cdots,4,\quad\mbox{and}\quad \beta_{i,\ell}=0~\mbox{for the other}~i,\ell.
\end{aligned}
\end{equation}
Then using (d1) and (d2), we get
\begin{equation*}
%\label{3.beta}
\beta_i:=2-\sum_{\ell=1}^{3n+1}\beta_{i,\ell}=b_{4i+1}.
\end{equation*}
While from (d1) and (d3), it is not difficult to see that
$$\beta_{1,1}<\beta_{1,2}+\beta_{1,3}+\beta_{1,4},$$
and
$$\beta_{i,j}<\sum_{\ell\neq j}\beta_{i,\ell},\quad i=2,\cdots,n.$$
Hence $\beta_{i,\ell}$ satisfies \eqref{1.ltcondition}.

Now let us state the main result of this section:

\begin{proposition}
\label{pr3.1}
Let $n\geq 2$ and $\beta_{i,\ell}$ be as in \eqref{3.con3} with $\{b_1,\cdots,b_{4n+1}\}$ satisfying the assumption $\mathcal{D}$, then there exists points $\{p_\ell\}_{\ell=1}^{3n+1}$ such that equation \eqref{1.u} (with the corresponding Lie algebra matrix  $\mathbf{A}_n$) has no solution satisfying the asymptotic condition \eqref{1.newasy}.
\end{proposition}

\begin{proof}
We shall apply the induction method to prove the conclusion. When $n=2$, the problem becomes
\begin{align*}
%\label{3.pr-1}
\left\{ \begin{array}{ll}
\Delta u_1+a_{1,1}K_1e^{2u_1}-a_{1,2}K_2e^{2u_2}=0\quad\text{in }\R^2\\ \rule{0cm}{0.5cm}
\Delta u_2-a_{2,1}K_1e^{2u_1}+a_{2,2}K_2e^{2u_2}=0 \quad\text{in }\R^2
\end{array} \right.
%\quad\mbox{in}\quad\B,
\end{align*}
where
\[K_1(x)=\prod_{i=1}^4\frac{1}{|x-p_i|^{2\beta_{1,i}}},\quad
K_2(x)=\prod_{i=5}^7\frac{1}{|x-p_i|^{2\beta_{2,i}}}.\]
One can check that $b_1,b_2,b_3,b_4,b_6,b_7,b_8$ satisfies the assumptions $\mathcal{A}1)$ to $\mathcal{A}5)$
(with $\beta_i=b_i$ for $i=1,2,3,4$ and $\beta_i=b_{i+1}$ for $i=5,6,7$), then by \cite[Lemma 3.2]{hlw}, we can find   points $p_1,\cdots,p_7$ such that \eqref{1.u} has no solution with the asymptotic behavior \eqref{1.newasy}.

Suppose the result holds for $n_0$ with $2\leq n_0$, and let $p_1,\cdots,p_{3n_0+1}$ be the points such that the following equation has no solution
\begin{align}
\label{3.m-1}
\left\{\begin{array}{ll}\Delta u_i+\sum_{j=1}^{n_0}a_{i,j}K_je^{2u_j}=0, &\quad i=1,\cdots,n_0,\\
\rule{0cm}{.5cm}
u_i(x)=-(2-\sum_{\ell=1}^3\beta_{i,3i-2+\ell})\log|x|~\mbox{as}~|x|\to\infty, &\quad i=1,\cdots,n_0.\end{array}\right.
\end{align}
Now let us find points  {$\{p_i\}_{i=3n_0+2}^{3n_0+4}$} so that   the conclusion holds for  $\{p_i\}_{i=1}^{3n_0+4}$. Let $p_{3n_0+2}$ be a fixed point (different from $p_1,\cdots,p_{3n_0+1}$). We claim that  for $|p_{3n_0+3}|,|p_{3n_0+4}|$ large and $p_{3n_0+3}\neq p_{3n_0+4}$  there exists no solution to \eqref{1.u} with $n=n_0+1$ having the asymptotic behavior
\begin{equation*}
u_i(x)=-\beta_i\log|x|+O(1)\quad\mbox{as}\quad|x|\to+\infty,~i=1,\cdots,n_0+1.
\end{equation*}
We shall prove the claim by contradiction. Suppose there is a sequence of solutions $\{u_i^k\}_{i=1}^{n_0+1}$ of \eqref{1.u}-\eqref{1.newasy} ($n$ is replaced by $n_0+1$) with
\begin{equation*}
p_{\ell}=p_{\ell,k},\quad |p_{\ell}|\to\infty\quad\mbox{for}\quad \ell=3n_0+3,3n_0+4.
\end{equation*}
Equivalently, we have $u_i^k,~i=1,\cdots,n_0+1$ satisfies
\begin{align*}
%\label{3.pr-2}
\left\{\begin{array}{ll}
\Delta u_i^k+\sum_{j=1}^{n_0+1}a_{i,j}\widetilde K_{j}e^{2u_j^k}=0 \quad\mbox{in} &\quad \B,\quad i=1,\cdots,n_0+1,\\
\rule{0cm}{.5cm}
\int_{\B}\w K_ie^{2u_i^k}=2\pi\bar\beta_i, &\quad i=1,\cdots,n_0+1,\\
\rule{0cm}{.5cm}
|p_{3n_0+3}|,|p_{3n_0+4}|\to+\infty,\end{array}\right.
\end{align*}
where $\w K_i(x)=K_i(x),~i=1,\cdots,n_0,$ and
\begin{equation*}
\w K_{n_0+1}(x)=|p_{3n_0+3}|^{2\beta_{n_0+1,3n_0+3}}|p_{3n_0+4}|^{2\beta_{n_0+1,3n_0+4}}
\prod_{\ell=1}^{3}\frac{1}{|x-p_{3n_0+1+\ell}|^{2\beta_{n_0+1,3n_0+1+\ell}}}.
\end{equation*}
Notice that $\w K_i,~i=1,\cdots,n_0$ is independent of $k$ and is integrable due to \eqref{3.con2}. We  see that
\begin{equation*}
\w K_{n_0+1}(x)\xrightarrow{k\to\infty}\frac{1}{|x-p_{3n_0+2}|^{2\beta_{n_0+1,3n_0+2}}}~\mbox{locally uniformly in}~\B\setminus\{p_{3n_0+2}\}.
\end{equation*}

We shall divide our proof into the following steps:

Step 1. We prove that
\begin{equation}
\label{3.pr-3}
\lim_{R\to\infty}\lim_{k\to\infty}\int_{B_R^c}\w K_ie^{2u_i}=0,\quad i=1,\cdots,n_0.
\end{equation}
We set
\begin{equation*}
%\label{3.pr-4}
\hat u_i^k(x)=u_i^k(\frac{x}{|x|^2})-\beta_i\log|x|,\quad i=1,\cdots,n_0+1,
\end{equation*}
then setting
\[q_{\ell}:=\frac{p_{\ell}}{|p_{\ell}|^2},\quad \ell=1,\cdots,3n_0+4,\]
we see that $\hat u_i^k$   satisfies
\begin{equation*}
%\label{3.pr-5}
\Delta\hat u_i^k(x)+\sum_{j=1}^{n_0+1}a_{i,j}\hat K_je^{2\hat u_{j}^k}=0,\quad i=1,\cdots,n_0,n_0+1,
\end{equation*}
where
\[\hat K_i(x)=\prod_{\ell=1}^{3n_0+1}\frac{|q_{\ell}|^{2\beta_{i,\ell}}}{|x-q_{\ell}|^{2\beta_{i,\ell}}},
\quad i=1,\cdots,n_0,\]
and
\[\hat K_{n_0+1}=|q_{3n_0+2}|^{2\beta_{n_0+1,3n_0+2}}
\prod_{\ell=1}^3\frac{1}{|x-q_{3n_0+1+\ell}|^{2\beta_{n_0+1,3n_0+1+\ell}}}.\]
We set
$$\bar\beta_i=\sum_{j=1}^{n_0+1}c^{i,j}\beta_j,\quad i=1,\cdots,n_0+1,$$
where $c^{i,j}$ is the inverse matrix of $a_{i,j}$ of rank $n_0+1$. Using Lemma \ref{le4.inverse}, we have
\begin{equation*}
\sum_{j=1}^{n_0+1}c^{i,j}<4(n_0+1)^2.
\end{equation*}
Combined with (d4), it is easy to check that
\begin{equation}
\label{3.pr-6}
\frac{1}{2\pi}\int_{\B}\hat K_ie^{2\hat u_i^k}
=\bar\beta_i<\frac{1}{100},\quad i=1,\cdots,n_0+1,
\end{equation}
and
\begin{equation*}
%\label{3.pr-7}
\hat K_i(x)\to 1\quad\mbox{as}\quad |x|\to0,\quad i=1,\dots,n_0.
\end{equation*}
Now we can apply Lemma \ref{le4.bm} with $\alpha=0$, to get
$$\hat u_i^k(x)\leq C~\mbox{in a neighborhood of origin},\quad i=1,\cdots,n_0.$$
Thus \eqref{3.pr-3} follows.\\

Step 2. Set
$$\mathfrak{S}_i=\left\{x\in\B:\mbox{there is a sequence}~x^k\to x~\mbox{such that}~u_i^k(x^k)\to\infty\right\},~i=1,\cdots,n_0+1.$$
We shall show that
$$\mathfrak{S}=\bigcup_{i=1}^{n_0+1}\mathfrak{S}_i=\emptyset.$$
For any $p\in\mathfrak{S},$ we set
\begin{equation}
\label{3-sigma-p}
\sigma_i(p)=\frac{1}{2\pi}\lim_{r\to0}\lim_{k\to\infty}\int_{B_r(p)}\w K_ie^{u_i^k}dx,\quad i=1,\cdots,n_0+1.
\end{equation}
It is well known that $u_i^k$ satisfies the integral equation
\begin{equation*}
%\label{3-integral}
u_i^k(x)=\frac{1}{2\pi}\int_{\B}\log\left(\frac{1+|y|}{|x-y|}\right)
\sum_{j=1}^{n_0+1}a_{i,j}\w K_je^{2u_j^k(y)}dy+C_{i}^k,~i=1,\cdots,n_0+1.
\end{equation*}
For any $p\in\mathfrak{S}$, let $r_0>0$ be such that $B_{r_0}(p)\cap\mathfrak{S}=\{p\}$, then from the above integral representation, one can show that
\begin{equation*}
|u_i^k(x)-u_i^k(y)|\leq C\quad\mbox{for every}\quad x,y\in\partial B_{r_0}(p),~\quad i=1,\cdots,n_0+1.
\end{equation*}
Therefore, in a neighborhood of each blow-up point, we see that $u_i^k$ satisfies the bounded oscillation property, and it implies that $\sigma_1(p),\cdots,\sigma_{n_0+1}(p)$ satisfies the pohozaev identity (see Lemma \ref{le4.1})
\begin{equation}
\label{3.poho}
\sum_{i=1}^{n_0+1}\sigma_i^2(p)-\sum_{i=1}^{n_0}\sigma_i(p)\sigma_{i+1}(p)=\sum_{i=1}^{n_0+1}\mu_i(p)\sigma_i(p),
\end{equation}
where
\begin{equation*}
\mu_i(p)=\begin{cases}
1,\quad &\mbox{if}~p\notin\{p_1,\cdots,p_{3n_0+1}\},\\
1-\beta_{i,\ell},\quad &\mbox{if}~p=p_{\ell}.
\end{cases}
\end{equation*}
Using \eqref{3.poho}, we conclude that for at least one index $i$ of $\{1,\cdots,n_0+1\}$, $\sigma_i(p)\geq\mu_i(p)$. As a consequence,
\begin{equation}
\label{3.2-1}
\bar\beta_i\geq \sigma_1(p)\geq \mu_i(p).
\end{equation}
For $p\in\B\setminus\{p_1\}$, we shall show \eqref{3.2-1} is impossible. Using \eqref{3.con2} and \eqref{3.pr-6}, we have
\begin{equation*}
%\label{3.2-2}
\begin{cases}
\beta_{1,\ell}+\bar\beta_1<1~\mbox{for}~\ell=2,3,4,\\
\beta_{i,3i-2+\ell}+\bar\beta_i<1~\mbox{for}~i=2,\cdots,n_0,~\ell=1,2,3,\\
\beta_{n_0+1,3n_0+2}+\bar\beta_{n_0+1}<1.
\end{cases}
\end{equation*}
This implies \eqref{3.2-1} never holds if $p\in\B\setminus\{p_1\}$. If $p=p_1$, then we can apply Lemma \ref{le4.bm} to conclude that $u_i^k,~i=2,\cdots,n_0+1$ are uniformly bounded above in a neighborhood of $p_1,$ otherwise, we would get
$$\bar\beta_i\geq\sigma_i(p)\geq 1,$$
which contradicts to (d4). Then we get
\begin{equation}
\label{3.2-3}
\bar\beta_1\geq\sigma_1(p)=1-\beta_{1,1}.
\end{equation}
In fact, {as $\beta_{1,4}=b_4=b_5=b_9=\beta_2$, $\beta_1=b_5$,} we have
\begin{equation}
\label{3.2-3-0}
\bar\beta_1=\sum_{j=1}^{n_0+1}c^{1,j}\beta_j\geq\frac{n_0+1}{n_0+2}\beta_{1,4}+\frac{n_0}{n_0+2}\beta_{1,4}
\geq \beta_{1,4},
\end{equation}
where we used $c^{i,j}=\frac{\min\{i,j\}(n_0+2-\max\{i,j\})}{n_0+2}$, see Lemma \ref{le4.inverse}. Using (d1) and \eqref{3.con1}, we have
\begin{equation*}
\bar\beta_1+\beta_{1,1}\geq b_1+b_4>1.
\end{equation*}
Therefore, the strict inequality of \eqref{3.2-3} holds, i.e.,
\begin{equation*}
%\label{3.2-3-1}
\bar\beta_1>\sigma_1(p)=1-\beta_{1,1}.
\end{equation*}
Then we can apply the arguments of \cite[Theorem 3]{Brezis-Merle} to get that concentration property holds for $u_1^k$, that is
\begin{equation*}
%\label{3.2-4}
u_1^k(x)\to-\infty\quad\mbox{locally uniformly in}~\B\setminus\mathfrak{S}_1,
\end{equation*}
we must have that the Cardinality of $\mathfrak{S}_1$ is at least 2, thanks to the Step 1. However, we have already shown that
$$\mathfrak{S}\setminus\{p_1\}=\emptyset.$$
Thus contradiction arises again, and $\mathfrak{S}=\emptyset.$\\

Step 3. $u_i^k\to \bar u_i$ in $C_{\mathrm{loc}}^2(\B),~i=1,\cdots,n_0$, where $u_1,\cdots,u_{n_0}$ satisfies \eqref{3.m-1}. Since
$$\mathfrak{S}=\emptyset\quad\mbox{and}\quad \bar\beta_i>0,~i=1,\cdots,n_0,$$
one of the following holds: passing to a subsequence if necessary,
\begin{enumerate}
  \item [(i)] $u_i^k\to\bar u_i$ in $C_{\mathrm{loc}}^2(\B)$ for $i=1,\cdots,n_0+1,$
  \item [(ii)] $u_i^k\to\bar u_i$ in $C_{\mathrm{loc}}^2(\B)$ for $i=1,\cdots,n_0,$ and $u_{n_0+1}^k\to-\infty$ locally uniformly in $\B$.
\end{enumerate}
Now we assume by contradiction that (i) happens. Then we get that the limit functions $(\bar u_1,\cdots,\bar u_{n_0+1})$ satisfy the system
\begin{equation*}
%\label{3.3-toda}
\left\{\begin{array}{ll}%\begin{cases}
\Delta\bar u_i+\sum_{j=1}^{n_0+1}a_{i,j}\bar K_je^{\bar u_j}=0~\mbox{in}~\B,\quad i=1,\cdots,n_0+1,\\ \rule{.0cm}{.5cm}
\int_{\B}\bar K_ie^{2\bar u_i}=2\pi\bar\beta_i,~i=1,\cdots,n_0,\quad
\int_{\B}\bar K_{n_0+1}e^{2\bar u_{n_0+1}}=2\pi\gamma\leq 2\pi\bar\beta_{n_0+1},\end{array}\right.
%\end{cases}
\end{equation*}
where
\[\bar K_i(x)=\w K_i(x),~i=1,\cdots,n_0,\quad\mbox{and}\quad \bar K_{n_0+1}=\frac{1}{|x-p_{3n_0+2}|^{2\beta_{n_0+1,3n_0+2}}}.\]
Then one has
\begin{equation*}
%\label{3.3-1}
\lim_{|x|\to\infty}\frac{\bar u_{n_0+1}(x)}{\log|x|}=-(2\gamma-\bar\beta_{n_0}),
\end{equation*}
and together with $\bar K_{n_0+1}e^{\bar u_{n_0+1}}\in L^1(\B)$ we have
\begin{equation}
\label{3.3-2}
\beta_{n_0+1,3n_0+2}+2\gamma-\bar\beta_{n_0}>1,
\end{equation}
and this is impossible due to (d2)-(d4) in assumption $\mathcal{D}$.
%{\color{blue}(is the contradiction trivial? I got it by computing the sum for $\bar \beta_{n_0}$ and $\bar\beta_{n_0+1}$)}.
Therefore, (ii) holds, and we get that it reduces to the equation \eqref{3.m-1}. However, \eqref{3.m-1} has no solution and contradiction arises. Thus, the conclusion holds also for $n_0+1$ and we finish the whole proof.
\end{proof}

\begin{proof}[Proof of Theorem \ref{th1.2} for $\mathbf{A}_n$]
This is a direct consequence of Proposition \ref{pr3.1}.
\end{proof}

\begin{proof}[Proof of Theorem \ref{th1.2} for other Lie algebras]
For $\mathbf{E}_6$, $\mathbf{E}_7$, $\mathbf{E}_8$, $\mathbf{B}_n,$ $\mathbf{C}_n,$ and $\mathbf{D}_n$ with $n\geq 3$, , we can derive the counterpart non-existence results from $\mathbf{A}_5$, $\mathbf{A}_6$, $\mathbf{A}_7$, $\mathbf{A}_{n-1}$ through almost the same argument of Proposition \ref{pr3.1}. Indeed, we use $\mathbf{D}_n,~n\geq 4$ ($\mathbf{D}_n$ only make sense for $n\geq3$ and $\mathbf{D}_3=\mathbf{A}_3$) as an example to explain it. Let $\beta_{i,\ell}$ be as in \eqref{3.con3} with $\{b_1,\cdots,b_{4n+1}\}$ satisfying the assumption $\mathcal{D}$. Using Proposition \ref{pr3.1}, we can find points $\{p_{\ell}\}_{\ell=1}^{3n-2}$ such that equation \eqref{1.u} has no solution satisfying the asymptotic behavior \eqref{1.newasy}. Then we prove the non-existence result by contradiction. Following the Step 1 of the proof of Proposition \ref{pr3.1}, we define the same sequence of solutions of \eqref{1.u} with $n_0+1$ replaced by $n$ and $\mathbf{A}_{n_0+1}$ by $\mathbf{D}_n$, and reach the same conclusion \eqref{3.pr-3} for first $n-1$ components. In Step 2, we prove that the blow-up phenomena can not happen. To show that the blow-up point $p\notin\mathbb{R}^2\setminus\{p_1\}$, the only thing we used is the Pohozaev identity and there exists at least one index $i\in\{1,\cdots,n\}$ such that $\sigma_i(p)\geq \mu_i(p)$. Using Lemma \ref{le4.1}, we see that it holds for general simple Lie algebra matrix. To show $p\neq p_1$, we only use \eqref{3.2-3-0} and it is easy to check that it holds also for the other cases by \eqref{4.abcd} and \eqref{4.efg}. In Step 3, we used \eqref{3.3-2} when we exclude the case that the limit function $\bar u_n$ can not be bounded uniformly. While for $\mathbf{D}_n$ case, we get
\begin{equation*}
\lim_{|x|\to+\infty}\frac{\bar u_n(x)}{\log|x|}=-(2\tilde\gamma-\bar\beta_{n-2}),\quad \tilde\gamma=\frac{1}{2\pi}\int_{\B}\bar K_{n}e^{2\bar u_n},
\end{equation*}
and $\bar K_{n}e^{\bar u_n}\in L^1(\mathbb{R}^2)$ implies that
\begin{equation*}
\beta_{n,3n-1}+2\tilde\gamma-\bar\beta_{n-2}>1.
\end{equation*}
Using (d2)-(d4) in the assumption $\mathcal{D}$, we can show that the above inequality is still not true and the case that $\bar u_n$ is uniformly bounded can also be excluded.

For $\mathbf{C}_2$ ($\mathbf{B}_2$ is equivalent to $\mathbf{C}_2$) and $\mathbf{G}_2$, we need to use a different tuple of numbers. Let us fix $b_1,\cdots,b_7\in(0,1)$ satisfies the following assumption:
\begin{equation*}
%\label{3.ass1}
\mbox{Assumption}~\mathcal{D}_1: \begin{cases}
\mbox{(d6)}~ &\sum_{\ell=1}^4b_{\ell}+b_4=2,~2b_1+b_4<2,\\
\mbox{(d7)}~ &\sum_{\ell=5}^7b_{\ell}+b_4=2,~b_5=b_6=b_7,\\
\mbox{(d8)}~ &b_2=b_3<\frac12b_1,~b_{4}<\frac{1}{1000}.
\end{cases}
\end{equation*}
A typical example of $(b_1,\cdots,b_7)$ satisfying assumption $\mathcal{D}_1$ is
\begin{equation*}
\begin{aligned}
&b_1=1-\frac23\ve,~b_2=\frac12-\frac23\ve,~b_3=\frac12-\frac23\ve,~b_4=\ve,\\
&b_5=b_6=b_7=\frac{2}{3}-\frac13\ve,\quad \ve<\frac{1}{1000}.
\end{aligned}
\end{equation*}
We set
\begin{equation*}
\label{3.bcg-2}
\beta_{1,\ell}=\begin{cases}
b_{\ell},\quad&\mbox{if}~\ell=1,2,3,4\\
0,&\mbox{if}~\ell=5,6,7
\end{cases},\qquad
\beta_{2,\ell}=\begin{cases}
0,\quad&\mbox{if}~\ell=1,2,3,4\\
b_{\ell},&\mbox{if}~\ell=5,6,7
\end{cases}.
\end{equation*}
Then we can follow the arguments of \cite[Lemma 3.1 and Lemma 3.2]{hlw} to find points $\{p_{\ell}\}_{\ell=1}^7$ such that \eqref{1.u} has no solution verifying the asymptotic behavior \eqref{1.newasy}.

%{\color{blue}Would this approach work?  For $F_4$: Fix points so that $A_2$ has no solution. Then choose points for $u_3$ and $u_4$ such that $F_4$ has no solution. Here we want to show as in Proposition 3.1 that $u_3, u_4\to \infty$ locally uniformly.... instead of just last component vanishing, we are making last two components vanishing  at the same time.   }

For $\mathbf{F}_4$ we can use the non-existence result of $\mathbf{A}_2$ to find points $\{p_{\ell}\}_{\ell=1}^{10}$ such that
\begin{equation}
\label{3.todaf4-3}
\left\{  \begin{array}{ll}%\begin{cases}
\Delta u_1+2K_1e^{2u_1}-K_2e^{2u_2}=0~&\mbox{in}~\B,\\ \rule{0cm}{.5cm}
\Delta u_2-K_1e^{2u_1}+2K_2e^{2u_2}-2K_3e^{2u_3}=0\quad &\mbox{in}~\B,\\ \rule{0cm}{.5cm}
\Delta u_3-K_2e^{2u_2}+2K_3e^{2u_3}=0~&\mbox{in}~\B,\\ \rule{0cm}{.5cm}
u_i(x)=-(2-\sum_{\ell=1}^m\beta_{i,\ell})\log|x|+O(1)~\mbox{as}~|x|\to\infty,\quad \quad  & i=1,2,3,  \end{array}\right.
%\end{cases}
\end{equation}
has no solution. By letting $\hat u_3=u_3+\frac12\log2$, we can make \eqref{3.todaf4-3} to a new system with a symmetric   coefficient matrix. In this case, we can derive the corresponding Pohozaev identity of \eqref{3.todaf4-3} from \cite[Proposition 3.1]{LWZ},
\begin{equation*}
\sum_{i=1}^2\sigma_i^2(p)+2\sigma_3^2(p)-\sigma_1(p)\sigma_2(p)-2\sigma_2(p)\sigma_3(p)
=\sum_{i=1}^2\mu_i\sigma_i(p)+2\mu_3\sigma_3(p),
\end{equation*}
where $\sigma_i(p)$ is defined in the same spirit of \eqref{3-sigma-p}. We can easily see that there exists at least one index $i\in\{1,2,3\}$ such that $\sigma_i(p)\geq\mu_i(p).$ Then we follow the proof of Proposition \ref{pr3.1} to deduce the non-existence result of \eqref{3.todaf4-3} for suitable points $\{p_{\ell}\}_{\ell=1}^{10}$. %\textcolor{blue}
{Based on the non-existence result of \eqref{3.todaf4-3}, we fix the points $\{p_{\ell}\}_{\ell=1}^{10}.$ Next, we repeat the argument of Proposition \ref{pr3.1} to derive the non-existence result of $\mathbf{F}_4$ Toda system by choosing appropriate points $\{p_{\ell}\}_{\ell=11}^{13}.$}
\end{proof}

\smallskip
\section{Some useful results}
In this section, we shall present several useful facts which are used in previous section. The first one is on the
matrices of the general simple Lie algebras:
%{\allowdisplaybreaks
\begin{align*}
&\mathbf{A}_n=:\scriptsize{\left(\begin{matrix}
2&-1&0&\cdots&0&0&0\\
-1&2&-1&\cdots&0&0&0\\
\vdots&\vdots&\vdots&\ddots&\vdots&\vdots&\vdots\\
0&0&0&\cdots&2&-1&0\\
0&0&0&\cdots&-1&2&-1\\
0&0&0&\cdots&0&-1&2
\end{matrix}\right)},\quad
\mathbf{B}_n=:\scriptsize{\left(\begin{matrix}
2&-1&0&\cdots&0&0&0\\
-1&2&-1&\cdots&0&0&0\\
\vdots&\vdots&\vdots&\ddots&\vdots&\vdots&\vdots\\
0&0&0&\cdots&2&-1&0\\
0&0&0&\cdots&-1&2&-2\\
0&0&0&\cdots&0&-1&2
\end{matrix}\right)},\\
&\mathbf{C}_n=:\scriptsize{\left(\begin{matrix}
2&-1&0&\cdots&0&0&0\\
-1&2&-1&\cdots&0&0&0\\
\vdots&\vdots&\vdots&\ddots&\vdots&\vdots&\vdots\\
0&0&0&\cdots&2&-1&0\\
0&0&0&\cdots&-1&2&-1\\
0&0&0&\cdots&0&-2&2
\end{matrix}\right)},\quad
\mathbf{D}_n=:\scriptsize{\left(\begin{matrix}
2&-1&0&\cdots&0&0&0\\
-1&2&-1&\cdots&0&0&0\\
\vdots&\vdots&\vdots&\ddots&\vdots&\vdots&\vdots\\
0&0&0&\cdots&2&-1&-1\\
0&0&0&\cdots&-1&2&0\\
0&0&0&\cdots&-1&0&2
\end{matrix}\right)},\\
&\mathbf{E}_6=:\scriptsize{\left(\begin{matrix}
2&-1&0&0&0&0\\
-1&2&-1&0&0&0\\
0&-1&2&-1&0&-1\\
0&0&-1&2&-1&0\\
0&0&0&-1&2&0\\
0&0&-1&0&0&2
\end{matrix}\right)},\quad
\mathbf{E}_7=:\scriptsize{\left(\begin{matrix}
2&-1&0&0&0&0&0\\
-1&2&-1&0&0&0&0\\
0&-1&2&-1&0&0&0\\
0&0&-1&2&-1&0&-1\\
0&0&0&-1&2&-1&0\\
0&0&0&0&-1&2&0\\
0&0&0&-1&0&0&2
\end{matrix}\right)},\\
&\mathbf{E}_8=:\scriptsize{\left(\begin{matrix}
2&-1&0&0&0&0&0&0\\
-1&2&-1&0&0&0&0&0\\
0&-1&2&-1&0&0&0&0\\
0&0&-1&2&-1&0&0&0\\
0&0&0&-1&2&-1&0&-1\\
0&0&0&0&-1&2&-1&0\\
0&0&0&0&0&-1&2&0\\
0&0&0&0&-1&0&0&2
\end{matrix}\right)},~
\mathbf{F}_4=:\scriptsize{\left(\begin{matrix}
2&-1&0&0\\
-1&2&-2&0\\
0&-1&2&-1\\
0&0&-1&2
\end{matrix}\right)},~
\mathbf{G}_2=:\scriptsize{\left(\begin{matrix}
2&-1\\
-3&2
\end{matrix}\right)}.
\end{align*}
%}
We shall derive an estimate on each entry of the inverse matrix of above matrices. For $\mathbf{A}_n,~\mathbf{B}_n,~\mathbf{C}_n$ and $\mathbf{D}_n$ type matrices, we get their inverse matrices as follows, see \cite[section 4]{RS} or \cite{wz}.
\begin{equation}
\label{4.abcd}
\begin{aligned}
&\left(\mathbf{A}_{n}^{-1}\right)_{i,j}=\frac{\min\{i,j\}(n+1-\max\{i,j\})}{n+1},~1\leq i,j\leq n.\\
&\left(\mathbf{B}_n^{-1}\right)_{i,j}=\begin{cases}
\min\{i,j\},~&1\leq i\leq n-1,~1\leq j\leq n,\\
\frac12j,~&i=n,~1\leq j\leq n.
\end{cases}\\
&\left(\mathbf{C}_n^{-1}\right)_{i,j}=\begin{cases}
\min\{i,j\},~&1\leq i\leq n,~1\leq j\leq n-1,\\
\frac12i,~&1\leq i\leq n,~j=n.
\end{cases}\\
&\left(\mathbf{D}_n^{-1}\right)_{i,j}=\begin{cases}
\min\{i,j\},~&1\leq i,j\leq n-2,\\
\frac12\min\{i,j\},~&1\leq\min\{i,j\}\leq n-2<\max\{i,j\}\leq n,\\
\frac14(n-2),~&i=n,j=n-1,~\mathrm{or}~i=n-1,j=n,\\
\frac{1}{4}n,~&i=n-1,j=n-1,~\mathrm{or}~i=n,j=n.
\end{cases}
\end{aligned}
\end{equation}
\noindent By straightforward computation, we have
\begin{equation}
\label{4.efg}
\begin{aligned}
&\mathbf{E}_6^{-1}=\scriptsize{\left(\begin{matrix}
4/3&5/3&2&4/3&2/3&1\\
5/3&{10/3}&4&8/3&4/3&2\\
2&4&6&4&2&3\\
4/3&8/3&4&10/3&5/3&2\\
2/3&4/3&2&5/3&4/3&1\\
1&2&3&2&1&2
\end{matrix}\right)},
\quad
\mathbf{E}_7^{-1}=\scriptsize{\left(\begin{matrix}
3/2&2&5/2&3&2&1&3/2\\
2&4&5&6&4&2&3\\
5/2&5&15/2&9&6&3&9/2\\
3&6&9&12&8&4&6\\
2&4&6&8&6&3&4\\
1&2&3&4&3&2&2\\
3/2&3&9/2&6&4&2&7/2
\end{matrix}\right)},\\
&\mathbf{E}_8^{-1}=\scriptsize{\left(\begin{matrix}
2&3&4&5&6&4&2&3\\
3&6&8&10&12&8&4&6\\
4&8&12&15&18&12&6&9\\
5&10&15&20&24&16&8&12\\
6&12&8&24&30&20&10&15\\
4&8&12&16&20&14&7&10\\
2&4&6&8&10&7&4&5\\
3&6&9&12&15&10&5&8
\end{matrix}\right)},~
\mathbf{F}_4^{-1}=\scriptsize{\left(\begin{matrix}
2&3&4&2\\
3&6&8&4\\
2&4&6&3\\
1&2&3&2
\end{matrix}\right)},~
\mathbf{G}_2^{-1}=\scriptsize{\left(\begin{matrix}
2&1\\
3&2
\end{matrix}\right)}.
\end{aligned}
\end{equation}
From \eqref{4.abcd} and \eqref{4.efg}, we get the following conclusion
\begin{lemma}
\label{le4.inverse}
For each type Cartan matrix $(a_{i,j})_{n\times n},$ let $(c_{i,j})_{n\times n}$ be its inverse matrix. Then we have the following estimate on $c_{i,j},~1\leq i,j\leq n,$
\begin{equation*}
%\label{4.est}
0<c_{i,j}<4n.
\end{equation*}
\end{lemma}
\medskip

The following lemma is a generalization of Brezis-Merle \cite{Brezis-Merle} type result, we refer the readers to \cite[Lemma 5.1]{hlw} for a proof.
\begin{lemma}
\label{le4.bm}
Let $u^k$ be a sequence of solutions to
\begin{equation*}
\left\{ \begin{array}{ll} %\begin{cases}
\Delta u^k+\frac{f^k(x)}{|x|^{2\alpha}}e^{2u^k}=g_k~\mbox{in}~B_1,\\ \rule{0cm}{.5cm}
\int_{B_1}\frac{f^k(x)}{|x|^{2\alpha}}{e^{2u^k}}dx\leq 2\pi(1-\alpha-\delta), \end{array}\right.
%\end{cases}
\end{equation*}
where $\delta>0$, $\alpha\in[0,1)$, and $g^k$ is a family of non-negative functions such that $\|g^k\|_{L^1(B_1)}\leq C$. Suppose that $0\leq f^k\leq C$ and $\inf_{B_1\setminus B_{\tau}}f^k\geq C_{\tau}$ for some $\tau\in(0,\frac13)$. Then $\{u^k\}$ is locally uniformly bounded from above in $B_1.$
\end{lemma}

The last result of this section is about the Pohozaev typde identity for singular Toda system. See \cite{LWYZ,LWZ,lyz} for related results.

\begin{lemma}
\label{le4.1}
Let $(u_1^k,\cdots,u_n^k)$ be a sequence of solutions to
\begin{equation*}
%\label{4.1}
\begin{cases}
\Delta u_i^k+\sum_{j=1}^na_{i,j}\frac{h_j^k}{|x|^{2\alpha_j}}e^{2u_j^k}=0~\mbox{in}~B_1,&i=1,\cdots,n,\\
\int_{B_1}\frac{h_i^k}{|x|^{2\alpha_i}}e^{2u_i^k}dx\leq C,\quad &i=1,\cdots,n,\\
|u_i^k(x)-u_i^k(y)|\leq C~\mbox{for every}~x,y\in\partial B_1,~&i=1,\cdots,n,\\
\|h_i^k(x)\|_{C^3(B_1)}\leq C,\quad 0<\frac1C\leq h_i^k(x)~\mathrm{in}~B_1,~&i=1,\cdots,n,
\end{cases}
\end{equation*}
for some $\alpha_i<1,~i=1,\cdots,n$ and $B_1$ is the unit ball in $\B$. Assume that $0$ is the only blow up point, that is,
\begin{equation*}
\sup_{B_1\setminus B_{\ve}} u_i^k(x)\leq C(\ve)~\ \mbox{for every}~\ 0<\ve<1,\quad i=1,\cdots,n.
\end{equation*}
Then setting $\mu_i=1-\alpha_i$ and
\begin{equation*}
\sigma_i:=\frac{1}{2\pi}\lim_{r\to0}\lim_{k\to\infty}\int_{B_r}\frac{h_i^k(x)}{|x|^{2\alpha_i}}e^{2u_i^k(x)}dx,
\quad i=1,\cdots,n,
\end{equation*}
we have
%{\allowdisplaybreaks
\begin{align*}
&\mathbf{A}_n:~\sum_{i=1}^n\sigma_i^2-\sum_{i=1}^{n-1}\sigma_i\sigma_{i+1}=\sum_{i=1}^n\mu_i\sigma_i,\\
&\mathbf{B}_n:~\sum_{i=1}^{n-1}\sigma_i^2+2\sigma_n^2-\sum_{i=1}^{n-2}\sigma_{i}\sigma_{i+1}-
2\sigma_{n-1}\sigma_n=\sum_{i=1}^{n-1}\mu_i\sigma_i+2\mu_n\sigma_n,\\
&\mathbf{C}_n:~2\sum_{i=1}^{n-1}\sigma_i^2+\sigma_n^2-2\sum_{i=1}^{n-1}\sigma_i\sigma_{i+1}=
2\sum_{i=1}^{n-1}\mu_i\sigma_i+\mu_n\sigma_n,\\
&\mathbf{D}_n:~\sum_{i=1}^n\sigma_i^2-\sum_{i=1}^{n-2}\sigma_{i}\sigma_{i+1}-\sigma_{n-2}\sigma_n=
\sum_{i=1}^n\mu_i\sigma_i,\\
&\mathbf{E}_n:~\sum_{i=1}^n\sigma_i^2-\sum_{i=1}^{n-2}\sigma_{i}\sigma_{i+1}-\sigma_{n-3}\sigma_n=
\sum_{i=1}^n\mu_i\sigma_i,~n=6,7,8,\\
&\mathbf{F}_4:~\sum_{i=1}^2\sigma_i^2+2\sum_{i=3}^4\sigma_i^2-\sigma_1\sigma_2
-2\sum_{i=2}^3\sigma_i\sigma_{i+1}=\sum_{i=1}^2\mu_i\sigma_i+2\sum_{i=3}^4\mu_i\sigma_i,\\
&\mathbf{G}_2:~3\sigma_1^2-3\sigma_1\sigma_2+\sigma_2^2=3\mu_1\sigma_1+\mu_2\sigma_2.
\end{align*}
%}
In particular, for each type Toda system, if $(\sigma_1,\cdots,\sigma_n)\neq(0,\cdots,0)$, then there exists at least one index $i\in\{1,\cdots,n\}$ such that
\begin{equation}
\label{4.comparison}
\sigma_i\geq \mu_i=1-\alpha_i.
\end{equation}
\end{lemma}

\begin{proof}
Since $\mathbf{A}_n$, $\mathbf{D}_n$, $\mathbf{E}_6,~\mathbf{E}_7,$ and $\mathbf{E}_8$ are symmetric matrix, we get the corresponding Pohozaev identity from \cite[Proposition 3.1]{LWZ} directly. For $\mathbf{B}_{n}$, $\mathbf{C}_{n}$ and $\mathbf{G}_2$, we derive their Pohozaev identities from $\mathbf{A}_{2n}$, $\mathbf{A}_{2n-1}$, $\mathbf{A}_6$ type Toda system respectively, see \cite[Lemma 4.1 and Lemma 4.2]{nie1} and \cite[Example 3.4]{nie2}. While for $\mathbf{F}_4$, let $\hat u_i=u_i,~i=1,2$ and $\hat u_i=u_i+\frac12\log2,~i=3,4$, we get $(\hat u_1,\hat u_2,\hat u_3,\hat u_4)$ satisfies
\begin{equation}
\label{4.modifiedf4}
\begin{cases}
\Delta \hat u_1^k+2\frac{h_1^k}{|x|^{2\alpha_1}}e^{2\hat u_1^k}-\frac{h_2^k}{|x|^{2\alpha_2}}e^{2\hat u_2^k}=0,\\
\Delta \hat u_2^k-\frac{h_1^k}{|x|^{2\alpha_1}}e^{2\hat u_1^k}+2\frac{h_2^k}{|x|^{2\alpha_2}}e^{2\hat u_2^k}
-\frac{h_3^k}{|x|^{2\alpha_3}}e^{2\hat u_3^k}
=0,\\
\Delta \hat u_3^k-\frac{h_2^k}{|x|^{2\alpha_2}}e^{2\hat u_2^k}+\frac{h_3^k}{|x|^{2\alpha_3}}e^{2\hat u_3^k}
-\frac12\frac{h_4^k}{|x|^{2\alpha_4}}e^{2\hat u_4^k}=0,\\
\Delta \hat u_4^k-\frac12\frac{h_3^k}{|x|^{2\alpha_3}}e^{2\hat u_3^k}
+\frac{h_4^k}{|x|^{2\alpha_4}}e^{2\hat u_4^k}=0.
\end{cases}
\end{equation}
We see that the coefficient matrix of \eqref{4.modifiedf4} is symmetric. By applying \cite[Proposition 3.1]{LWZ}, we get the related Pohozaev identity.

We shall prove \eqref{4.comparison} for $\mathbf{A}_n$ only, the other cases can be proved similarly. %{\color{red}Suppose \eqref{4.comparison} is not true for all $i$, then $\sigma_i<\mu_i$ holds for any $i\in\{1,2,\cdots,n\}$. As a consequence
%\begin{equation*}
%\sum_{i=1}^n\sigma_i^2-\sum_{i=1}^{n-1}\sigma_i\sigma_{i+1}<\sum_{i=1}^n\mu_i\sigma_i,
%\end{equation*}
%and it contradicts to the Pohozaev identity. Thus the original conclusion holds and we finish the proof.}
The Pohozaev identity for $A_n$ can be written as $$\sum_{i=1}^n\sigma_i(\sigma_i-\mu_i)=\sum_{i=1}^{n-1}\sigma_i\sigma_{i+1}\geq 0.$$  This shows that \eqref{4.comparison} holds for at least  one index $i\in \{1,\dots, n\}$.
\end{proof}


\bigskip
\begin{thebibliography}{10}
\small

%\bibitem{Adi}\textsc{Adimurthi, K. Sandeep:}  \emph{A singular Moser-Trudinger embedding and its applications}, NoDEA Nonlinear Differential Equations Appl. \textbf{13} (2007), no. 5-6, 585-603.

%\bibitem{AW} \textsc{W. Ao, L. Wang}, \emph{ New concentration phenomena for SU(3) Toda system,} J. Differential Equations \textbf{256} (2014) 1548-1580.


%\bibitem{Aviles} \textsc{P. Aviles:} \emph{Conformal complete metrics with prescribed nonnegative Gaussian curvature in $\mathbb{R}^2$}, Invent. Math. \textbf{83} (1986), no. 3, 519-544.

\bibitem{bfr} \textsc{J. Balog, L. F\'{e}her, L. O'Raifeartaigh:} \emph{Toda theory and $\mathcal{W}$-algebra from a gauged WZNW point of view,}  Ann. Physics, {\bf 203} (1990) 76-136.

%\bibitem{Bar-Lin}  \textsc{D. Bartolucci, A.  Jevnikar, C. S.  Lin:} \emph{Non-degeneracy and uniqueness of solutions to singular mean field equations on bounded domains}, J. Differential Equations \textbf{266} (2019), no. 1, 716-741.

%\bibitem{BJMR} \textsc{L. Battaglia, A. Jevnikar, A. Malchiodi, D. Ruiz,} \emph{ A general existence result for the Toda system on compact surfaces,} Adv. Math. \textbf{285 } (2015) 937-979.

%\bibitem{BMM} \textsc{D. Bartolucci, F. De Marchis, A.  Malchiodi:} \emph{Supercritical conformal metrics on surfaces with conical singularities},  Int. Math. Res. Not. IMRN \textbf{2011}, (2011) no. 24, 5625-5643.

%\bibitem{Bar-Tar} \textsc{D. Bartolucci, G. Tarantello:} \emph{The Liouville equation with singular data: a concentration-compactness principle via a local representation formula}, J. Differential Equations \textbf{185} (1) (2002) 161-180.

%\bibitem{Bar-Tar2} \textsc{D. Bartolucci, G. Tarantello:}  \emph{Liouville type equations with singular data and their applications to periodic multivortices for the electroweak theory}, Comm. Math. Phys. \textbf{229} (2002), no. 1, 3-47.

%\bibitem{BM}\textsc{L. Battaglia, A. Malchiodi:} \emph{Existence and non-existence results for the $SU(3)$ singular Toda system on compact surfaces}, J. Functional Analysis \textbf{270} (2016), 3750-3807.

%\bibitem{Bolton} \textsc{J. Bolton, L.M. Woodward:} \emph{Some geometrical aspects of the $2$-dimensional Toda equations}, in: Geometry, Topology and Physics, Campinas, (1996), de Gruyter, Berlin, 1997, pp.69-81

\bibitem{Brezis-Merle}\textsc{H. Brezis, F. Merle:} \emph{Uniform estimates and blow-up behaviour for solutions of $-\Delta u=V(x)e^u$ in two dimensions}, Comm. Partial Differential Equations \textbf{16} (1991), 1223-1253.


%\bibitem{Calabi}\textsc{E. Calabi:} \emph{Isometric imbedding of complex manifolds}, Ann. of Math.  \textbf{58} no. 2, (1953), 1-23.




%\bibitem{Carlotto}\textsc{A. Carlotto, A. Malchiodi:} \emph{Andrea A class of existence results for the singular Liouville equation} C. R. Math. Acad. Sci. Paris \textbf{349} (2011), no. 3-4, 161-166.


\bibitem{CL} \textsc{W. Chen, C. Li:} \emph{Classification of solutions of some nonlinear elliptic equations}, Duke Math. J. \textbf{63} (3) (1991), 615-622.

%\bibitem{WXChen}\textsc{W. X. Chen:}  \emph{A Tr\"udinger inequality on surfaces with conical singularities}, Proc. Amer. Math. Soc. \textbf{108} (3) (1990) 821-832.

%\bibitem{CW} \textsc{S. S. Chern, J.G. Wolfson:}\emph{Harmonic maps of the two-sphere into a complex Grassmann manifold. II}, Ann. of Math.  \textbf{125} no. 2,  (1987) 301-335.

%\bibitem{DPR} \textsc{D'Aprile, Teresa; Pistoia, Angela; Ruiz, David} \emph{ Asymmetric blow-up for the SU(3) Toda system.} J. Funct. Anal. \textbf{271} (2016), no. 3, 495-531.

\bibitem{Doliwa}\textsc{A. Doliwa:}  \emph{Holomorphic curves and Toda systems} Lett. Math. Phys. 39(1), 21-32 (1997).

\bibitem{Dunne} \textsc{G. Dunne:}  \emph{Self-dual Chern-Simons Theories},  Lecture Notes in Physics. Springer, Berlin (1995).

\bibitem{Eremenko} \textsc{A. Eremenko:} \emph{Metrics of positive curvature with conic singularities on the sphere}, Proc. Amer. Math. Soc. \textbf{132} (2004), no. 11, 3349-3355.

\bibitem{frrtw} \textsc{L. F\'{e}her, L. OˇRaifeartaigh, P. Ruelle, I. Tsutsui and A. Wipf:} \emph{Generalized Toda theories and $\mathcal{W}$-algebras associated with integral gradings,}  Ann. Physics, {\bf 213} (1992) 1-20.

\bibitem{gh} \textsc{P. Griffiths, J. Harris:} \emph{Principles of algebraic geometry,} John Wiley, (2014).

\bibitem{M-Guest} \textsc{M. Guest:} \emph{Harmonic maps, loop groups, and integrable systems}, London Mathematical Society Student Texts, vol.38, Cambridge University Press, Cambridge, 1997.

%\bibitem{H-class} \textsc{A. Hyder:} \emph{Structure of conformal metrics on $\R^n$ with constant $Q$-curvature}, to appear in Differential and Integral Equations (2019), arXiv: 1504.07095 (2015).

\bibitem{hlw} \textsc{A. Hyder, C.S. Lin, J.C. Wei:} \emph{On $SU(3)$ Toda system with multiple singular sources}, preprint.

\bibitem{HMM} \textsc{A. Hyder, G. Mancini, L. Martinazzi:} \emph{Local and nonlocal singular Liouville equations in Euclidean spaces}, arXiv: 1808.03624 (2018).

%\bibitem{JLW}\textsc{J. Jost, C.S. Lin, G. Wang:} \emph{Analytic aspects of the Toda System II: Bubbling behavior and existence of solutions}, Comm. Pure Appl. Math. \textbf{59} (2006), no. 4, 526-558.

%\bibitem{JW} \textsc{J. Jost,  G.  Wang:} \emph{Classification of solutions of a Toda system in $\R^2$}, Int. Math. Res. Not. \textbf{6} (2002), 277-290.

%\bibitem{Lin} \textsc{C.S. Lin:} \emph{A classification of solutions of conformally invariant fourth order equations in $\R^{n}$}, Comm. Math. Helv \textbf{73} (1998), 206-231.

\bibitem{LNW} \textsc{C. S. Lin, Z. Nie,  J. Wei:} \emph{Toda system and hypergeometric equations}, Transcations of American Math Society \textbf{370} (2018), no. 11, 7605-7626.

\bibitem{LNW1} \textsc{C. S. Lin, Z. Nie,  J. Wei:} \emph{Classification of solutions to general Toda systems with singular sources}, preprint.

\bibitem{LWYZ} \textsc{C. S. Lin, J.C.  Wei, W.  Yang, L.   Zhang:} \emph{On rank-$2$ Toda systems with arbitrary singularities: local mass and new estimates} Anal. PDE \textbf{11} (2018), no. 4, 873-898.

\bibitem{LWY} \textsc{C. S. Lin, J.C. Wei, D. Ye:} \emph{Classification and nondegeneracy of $SU(n+1)$ Toda system with singular sources}, Invent. Math. \textbf{190} (2012), no. 1, 169-207.

\bibitem{LWZ} \textsc{C. S. Lin,  J.C.  Wei,  L.  Zhang:}  \emph{Classification of blowup limits for $SU(3)$ singular Toda systems},  Anal. PDE \textbf{8} (2015), no. 4, 807-837.

\bibitem{lyz} \textsc{C. S. Lin, W. Yang, X. Zhong:} \emph{Apriori Estimates of Toda systems, I: the types of ${A}_n,{B}_n,{C}_n$ and ${G}_2$}, to appear in J. Differential Geometry.

\bibitem{Lucia-Nolasco} \textsc{M. Lucia, M. Nolasco:} \emph{$SU(3)$ Chern-Simons vortex theory and Toda systems}, J. Diff. Equations \textbf{184} (2002),  443-474.

\bibitem{Luo-Tian} \textsc{F. Luo, G. Tian:} \emph{Liouville equation and spherical convex polytopes}, Proc. Amer. Math. Soc. \textbf{116} (1992) no. 4, 1119-1129.

%\bibitem{Malchiodi-Ruiz}\textsc{A. Malchiodi, D. Ruiz:} \emph{New improved Moser-Trudinger inequalities and singular Liouville equations on compact surfaces}, Geom. Funct. Anal. \textbf{21} (2011), no. 5, 1196-1217.

%\bibitem{MarClass} \textsc{L. Martinazzi,}  \emph{Classification of solutions to the higher order Liouville's equation on $\mathbb{R}^{2m}$}, Math. Z. \textbf{263} (2009), 307-329.%

\bibitem{MPW}\textsc{M. Musso, A.  Pistoia, J.C. Wei:} \emph{New blow-up phenomena for $SU(n+1)$ Toda system}, J. Differential Equations \textbf{260} (2016), no. 7, 6232-6266.

\bibitem{nie1}\textsc{Z. H. Nie:}\emph{Classification of solutions to Toda system of types $C$ and $B$ with singular sources}, Calc. Var. Partial Differential Equations, 55(2016) no.3 23pp.

\bibitem{nie2}\textsc{Z. H. Nie:}\emph{On characteristic integrals of Toda field theories}, J. Nonlinear Math. Phys. 21 (2014) no.1 120-131.


\bibitem{NT}\textsc{M. Nolasco, G. Tarantello:}\emph{Vortex condensates for the $SU(3)$ Chern-Simons theory}, Commun. Math. Phys. 213(3), 599-639 (2000).

\bibitem{RS}\textsc{A. V. Razumov, M.V. Saveliev:}\emph{Lie algebras, geometry, and Toda-type systems}, Cambridge University Press, (1997).


\bibitem{PT} \textsc{J. Prajapat, G. Tarantello:} \emph{ On a class of elliptic problems in $\R^2$: symmetry and uniqueness results}, Proc. Royal Soc. Edinburgh \textbf{131 A} (2001), 967-985.

\bibitem{Ta} \textsc{G. Tarantello:} \emph{Multiple condensate solutions for the Chern-Simons-Higgs theory}, J. Math. Phys. \textbf{37} (1996) 3769-3796.

\bibitem{Tr} \textsc{M. Troyanov}, \emph{Prescribing curvature on compact surfaces with conical singularities}, Trans.  Am. Math. Soc. \textbf{324}  (1991) 793-821.

\bibitem{Tr2}\textsc{M. Troyanov:} \emph{Metric of constant curvature on a sphere with two conical singularities}, in Differential geometry,  Lect. Notes in Math., vol. 1410, Springer-Verlag, 1989, pp. 296-306.

\bibitem{Umehara} \textsc{M. Umehara, K. Yamada:} \emph{Metrics of constant curvature $1$ with three conical singularities on the $2$-sphere}, Illinois J. Math. \textbf{44} (2000), no. 1, 72-94.

%\bibitem{WY} \textsc{J. Wei,  D. Ye}, \emph{Nonradial solutions for a conformally invariant fourth order equation in $\mathbb{R}^3$,} Calc. Var. Partial Differential Equations \textbf{32} (2008), no. 3, 373-386.


\bibitem{wz}\textsc{Y. J. Wei, Y.M. Zou:} \emph{Inverses of Cartan matrices of Lie algebras and Lie superalgebras,} Linear Algebra and its Applications, 521 (2017), 283-298.

\bibitem{Yang}\textsc{Y. Yang:}  \emph{The relativistic non-abelian Chern-Simons equation}, Commun. Phys. 186(1),
199-218 (1999).

\bibitem{Yang1}\textsc{Y. Yang:} \emph{Solitons in Field Theory and Nonlinear Analysis,} Springer Monographs in Mathematics, Springer, New York (2001)

\end{thebibliography}
\end{document}